\documentclass[preprint]{elsarticle}

\pdfoutput=1

\usepackage[T1]{fontenc}
\usepackage{lineno,hyperref}
\usepackage{
  amssymb,                                                                      
  bm,
  graphicx,                                                                     
  mathtools,                                                                    
  diffcoeff,                                                                    
  wrapfig
}
\usepackage[extdef]{delimset}                                                   

\newcommand{\ccushion}{{\text{\tiny{|}} c}}
\newcommand{\ctable}{{\underline{c}}}

\newcommand{\rvline}{\hspace*{-\arraycolsep}\vline\hspace*{-\arraycolsep}}

\newtheorem{theorem}{Theorem}

\newtheorem{lemma}{Lemma}
\newdefinition{definition}{Definition}
\newdefinition{remark}{Remark}
\newdefinition{example}{Example}
\newproof{proof}{Proof}

\modulolinenumbers[5]

\journal{ }

\makeatletter
\def\ps@pprintTitle{%
 \let\@oddhead\@empty
 \let\@evenhead\@empty
 \def\@oddfoot{}%
 \let\@evenfoot\@oddfoot}
\makeatother

\begin{document}

\begin{frontmatter}

\title{On generalized Coulomb-Amontons' law in the context of rigid body dynamics}

\author{A.S. Vaganian}
\ead{armay@yandex.ru}
\address{Herzen State Pedagogical University of Russia \\
48 Moyka Embankment (Naberezhnaya r. Moyki) St. Petersburg, 191186, Russia}

\begin{abstract}
    A generalization of Coulomb-Amontons' law of dry friction recently proposed by V.~V.~Kozlov
    is considered in the context of rigid body dynamics.
    Universal requirements for dry friction tensor formulated by V.~V.~Kozlov
    are complemented by a condition taking into account the contact nature of dry friction,
    and applied to several models.
    For the famous Painleve problem a generalized Coulomb-Amontons' force without singularities,
    yet such that the dissipation takes place only at the point of contact, is found.
    By the example of the motion of a rigid ball on a plane with a single point of contact,
    it is shown that these principles are consistent
    with the well-known equations, studied by G.-G.~Coriolis.
    Further, a ball simultaneously touching two perpendicular planes at two points of contact is considered.
    The corresponding equations of motion are derived and analyzed.
    An exact particular solution that describes a technique used in practice in billiards is obtained.
    It is shown that unlike the single contact case, the Lagrange multipliers can depend on friction coefficients.
    With this in mind, a generalization of the conditions on the tensor of dry friction for the case of
    an arbitrary number of contacts is proposed.
\end{abstract}


\end{frontmatter}


\section{Introduction} \label{sec1}

In article~\cite{Kozlov2010855} V.~V.~Kozlov suggested a generalization of Coulomb-Amontons' law of dry friction
for constrained Lagrangian systems that eliminates singularities in friction forces and constraint reactions
and applied it to an arbitrary motion of a rigid body on a fixed surface
in case, where at any moment the contact between the body and the surface takes place at exactly one point.
According to V.~V.~Kozlov, the generalized Coulomb-Amontons' force of dry friction
acting on a body moving on a surface given by an equation $f(x) = 0$, $x \in \mathbb{R}^n$,
is defined as follows (see~\cite[formula~(2.6)]{Kozlov2010855}):
\begin{equation} \label{eq1}
    F = - \abs{R} \frac{\Phi \dot{x}}{\abs{\dot{x}}},
\end{equation}
where we denote
\[
    R = \biggl\{ \lambda \frac{\partial f}{\partial x^i} \biggr\}, \quad
    |R|^2 = \lambda^2 \biggl(
        \frac{\partial f}{\partial x},
        A^{-1} \frac{\partial f}{\partial x}
    \biggr), \quad
    A = \biggl\{\frac{\partial^2 T}{\partial \dot{x}^i \partial \dot{x}^j}\biggr\}, \quad
    |\dot{x}|^2 = (A \dot{x}, \dot{x}).
\]
Here the parentheses denote convolution of a covector with a vector,
$R$ is the constraint reaction, $\lambda$ is the Lagrange multiplier,
$T$ is the kinetic energy of the system, and $\Phi = \{\Phi_{ij}\}$ is a matrix, for all $x$ and $t$
satisfying the following two conditions:
\begin{enumerate}[i.]
    \item \label{cond1}
    (see~\cite[formula~(2.9)]{Kozlov2010855})
    \[
        (\Phi^T A^{-1}) \frac{\partial f}{\partial x} = \rho \frac{\partial f}{\partial x},
        \quad \rho \in \mathbb{R}.
    \]
    This restriction means mutual orthogonality of the generalized forces of dry friction
    and normal support reaction in the metrics, defined by matrix $A$:
    \[\begin{split}
        \biggl(\frac{\partial f}{\partial x}, A^{-1} F\biggr)                                           &
        = - \frac{|R|}{|\dot{x}|} \biggl(\frac{\partial f}{\partial x}, A^{-1} \Phi \dot{x}\biggr)   \\ &
        = - \frac{|R|}{|\dot{x}|} \biggl(\Phi^T A^{-1} \frac{\partial f}{\partial x}, \dot{x}\biggr)
        = - \rho \frac{|R|}{|\dot{x}|} \biggl(\frac{\partial f}{\partial x}, \dot{x}\biggr)
        = 0.
    \end{split}\]
    The last equality is a consequence of the constraint equation.
    Also, as shown in~\cite{Kozlov2010855}, this condition implies that multiplier $\lambda$
    does not depend on $\Phi$ and can be found from the equations of motion without friction.
    \item \label{cond2}
    (see~\cite[formula~(2.8)]{Kozlov2010855})
    \[
        (\Phi \dot{x}, \dot{x}) \ge 0.
    \]
    This condition expresses dissipativity of the friction force:
    \[
        (F, \dot{x}) = - \frac{|R|}{|\dot{x}|} (\Phi \dot{x}, \dot{x}) \le 0.
    \]
\end{enumerate}
From a geometric point of view, matrices $A$ and $\Phi$ are components of twice covariant tensor fields
in the configuration space --- tensor of inertia and \emph{tensor of dry friction} accordingly,
which implies invariance of conditions~\ref{cond1} and~\ref{cond2} with respect to changes of
the generalized coordinates.

In the presented work some of the results from~\cite{Kozlov2010855} are refined under assumption
of absolute rigidity of the interacting bodies.
This assumption, as will be shown, leads to additional restrictions on matrix $\Phi$
and friction force $F$, and can be formulated in a simple universal form.
As a consequence, for the Painleve falling rod problem and the problem of a rigid body rolling on a fixed plane,
we obtain much simpler formulas for the tensor of dry friction than in~\cite[\S\S~5--6]{Kozlov2010855}.
It is worth noting, that rigidity of the considered objects is mentioned in the formulations of
these problems in~\cite{Kozlov2010855} as well, but it is not reflected in conditions~\ref{cond1}--\ref{cond2}
themselves, because of what, among others, the components responsible for rolling and pivoting friction
appear in the corresponding expressions.
This may lead to a contradiction when the point of application of these forces, which coincides with
the point of contact in case of rigid bodies, is still, and their power is non-zero.
In this regard, see also comments by V.~F.~Zhuravlev~\cite{Zhuravlev2011147}.
Next, we show a counterexample to the hypothesis introduced in~\cite[\S~7]{Kozlov2010855},
that in case of multiple points of contact as well, the Lagrange multipliers do not depend on friction
coefficients.
Along the way, we derive equations that describe a real technique used in practice in billiards.
Finally, we propose a generalization of condition~\ref{cond1} for the case of an arbitrary number of contacts.

\section{Painleve problem} \label{sec2}

\paragraph{Problem setting and support reaction in the absence of friction}

Consider a mechanical system consisting of two material points $M_1$ and $M_2$ of mass $m_1$ and $m_2$,
connected to each other by a weightless rigid rod of length $l$ (see Fig.~\ref{fig1}).
Line $Ox$ constrains the motion of point $M_1$ to half-plane $y \ge 0$.
\begin{figure}[!ht]
    \centering\includegraphics{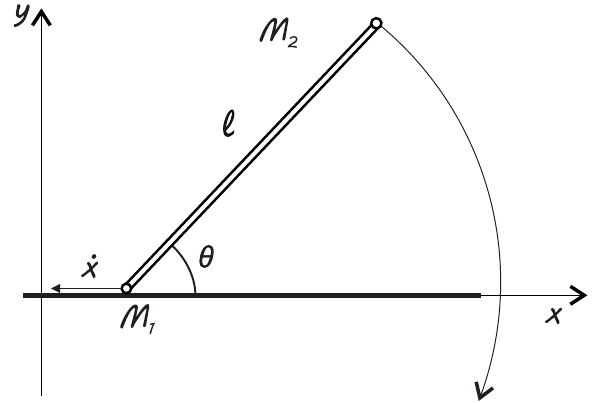}
    \caption{Painleve problem}
    \label{fig1}
\end{figure}
Lagrangian of the described system in generalized coordinates $x, \theta, y$ has the form
\[\begin{split}
    L & = m_1 \frac{\dot{x}^2 + \dot{y}^2}{2}
      + m_2 \frac{(\dot{x} - \dot{\theta} l \sin{\theta})^2 + (\dot{y} + \dot{\theta} l \cos{\theta})^2}{2} \\
      & - (m_1 + m_2) g y
      - m_2 g l \sin{\theta}
      + \lambda y,
\end{split}\]
where $x, y$ denote the Cartesian coordinates of point $M_1$,
$\theta$ is the angle between rod $M_1 M_2$ and axis $Ox$
(positive direction of the reference angle is counterclockwise),
$g$ is the acceleration of gravity, and $\lambda$ is the Lagrange multiplier.
Write down the equations of motion:
\begin{equation}\begin{alignedat}{3} \label{eq2}
    & (m_1 + m_2) \ddot{x} - l m_2 \ddot{\theta} \sin{\theta}
    && = l m_2 \dot{\theta}^2 \cos{\theta},                                     \\
    & l m_2 (\ddot{x} \sin{\theta} - \ddot{y} \cos{\theta} - l \ddot{\theta})
    && = l m_2 g \cos{\theta},                                                  \\
    & (m_1 + m_2) \ddot{y} + l m_2 \ddot{\theta} \cos{\theta}
    && = - (m_1 + m_2) g + l m_2 \dot{\theta}^2 \sin{\theta} + \lambda.         \\
\end{alignedat}\end{equation}
Depending on the value of $\lambda$, the motion of the system can be attributed to one of the following modes:
\begin{itemize}
    \item $\lambda = 0$, there is no contact between the rod and the support, i.e. $y > 0$.
    \item $\lambda \in (0, +\infty)$, there is a long-time contact of the rod with the support.
    In this case, letting $y$ together with derivatives equal to zero, we find
    \begin{equation} \label{eq3}
        \lambda = \frac{
            2 m_1 \bigl((m_1 + m_2) g - l m_2 \dot{\theta}^2 \sin{\theta}\bigr)
        }{
            2 m_1 + m_2 (1 + \cos{2 \theta})
        } > 0.
    \end{equation}
    \item $\lambda = +\infty$, there is an impact of the rod on the support,
    during which the generalized velocities undergo a discontinuity.
\end{itemize}

\paragraph{Painleve paradox}

The classical \emph{Coulomb-Amontons' law of dry friction} for the considered system is defined by expression
\begin{equation} \label{eq4}
    F = -\lambda\, \mu\, \sigma(\dot{x}), \qquad
    \sigma(x) = \begin{cases}
        -1,      & x < 0, \\
        [-1, 1], & x = 0, \\
        +1,      & x > 0, \\
    \end{cases}
\end{equation}
where $\mu > 0$ is the coefficient of friction, and $[-1, 1]$ on the right-hand side
means that $\sigma(0)$ can take any value from this range.
Following P.~Painleve (see~\cite[p.~16, formula~(d')]{Painleve1895141}),
let us add to the right-hand side of the Lagrange equation from system~\eqref{eq2}
corresponding to coordinate~$x$ force~$F$ to obtain a system with friction:
\begin{equation}\label{eq5}\begin{alignedat}{3}
    & (m_1 + m_2) \ddot{x} - l m_2 \ddot{\theta} \sin{\theta}
    && = l m_2 \dot{\theta}^2 \cos{\theta} + F,                                 \\
    & l m_2 (\ddot{x} \sin{\theta} - \ddot{y} \cos{\theta} - l \ddot{\theta})
    && = l m_2 g \cos{\theta},                                                  \\
    & (m_1 + m_2) \ddot{y} + l m_2 \ddot{\theta} \cos{\theta}
    && = - (m_1 + m_2) g + l m_2 \dot{\theta}^2 \sin{\theta} + \lambda,         \\
\end{alignedat}\end{equation}
Putting in these equations $y$ with derivatives equal to zero, we get
\[
    \lambda = \frac{
        2 m_1 \bigl((m_1 + m_2) g - l m_2 \dot{\theta}^2 \sin{\theta}\bigr)
    }{
        2 m_1 + m_2 (1 + \cos{2 \theta}) + \mu m_2 \sigma(\dot{x}) \sin{2 \theta}
    }.
\]
As can be seen, for sufficiently large values of the friction coefficient the denominator may turn to zero,
leading under a continuous contact to infinite values of the support reaction and friction forces,
which is devoid of physical sense.

\paragraph{V.~V.~Kozlov's approach to accounting for friction}

As noted in~\cite[\S5]{Kozlov2010855}, condition~\ref{cond1} is violated for system~\eqref{eq5},
which leads to singularities in $\lambda$.
Indeed, for the Painleve problem
\begin{equation} \label{eq6}
    A = \begin{pmatrix}
        m_1 + m_2               & - l m_2 \sin{\theta}  & 0                     \\
        - l m_2 \sin{\theta}    & m_2 l^2               & l m_2 \cos{\theta}    \\
        0                       & l m_2 \cos{\theta}    & m_1 + m_2             \\
    \end{pmatrix}, \quad
    f = y \ge 0
\end{equation}
so for $\Phi = \mathrm{diag}\{\varkappa, 0, 0\}$, we have
\[
    \Phi^T A^{-1} \begin{pmatrix}
        0 \\
        0 \\
        1 \\
    \end{pmatrix} = - \frac{ m_2 \varkappa \sin{2 \theta} }{ 2 m_1 (m_1 + m_2) } \begin{pmatrix}
        1 \\
        0 \\
        0 \\
    \end{pmatrix} \ne \rho \begin{pmatrix}
        0 \\
        0 \\
        1 \\
    \end{pmatrix}
\]
On the contrary, if condition~\ref{cond1} is satisfied, namely when (see~\cite[formulas~(5.1)--(5.2)]{Kozlov2010855})
\begin{equation} \label{eq7}
    \Phi = A \Omega^T, \quad
    \Omega = \begin{pmatrix}
        \varkappa_1 & \mu_1 & 0    \\
        \varkappa_2 & \mu_2 & 0    \\
        \varkappa_3 & \mu_3 & \nu  \\
    \end{pmatrix},
\end{equation}
then Coulomb-Amontons' dry friction force~\eqref{eq1} is uniquely defined by~\eqref{eq6}, \eqref{eq7} and~\eqref{eq3}
at any moment in time and any point of the phase space.

\paragraph{Insufficiency of conditions~\ref{cond1}-\ref{cond2} in case of absolutely rigid bodies}

Let in~\eqref{eq7} $\Omega = \nu E_3$, $\nu > 0$, where $E_3$ is the $3 \times 3$ identity matrix,
which corresponds to an isotropic tensor of dry friction.
In this case both conditions~\ref{cond1}-\ref{cond2} are met, and friction force~\eqref{eq1}
for $y = 0$ and $\lambda$ from~\eqref{eq3} is
\begin{equation} \label{eq8}
    F = - k \nu \begin{pmatrix}
        (m_1 + m_2) \dot{x} - l m_2 \dot{\theta} \sin{\theta} \\
        l m_2 (l \dot{\theta} - \dot{x} \sin{\theta})         \\
        l m_2 \dot{\theta} \cos{\theta}                       \\
    \end{pmatrix},
\end{equation}
where
\[
    k = \frac{
        \bigl((m_1 \! + \! m_2) g - l m_2 \dot{\theta}^2 \sin{\theta}\bigr) \sqrt{2 m_1}
    }{
        \sqrt{
            \! (m_1 \! + \! m_2) \bigl(2 m_1 \! + \! m_2 (1 + \cos{2 \theta})\bigr)
            \bigl(
                (m_1 \! + \! m_2 \cos^2{\theta}) \dot{x}^2 \! + \!
                m_2 (\dot{x} \sin{\theta} - l \dot{\theta})^2
            \bigr)
        }
    }
\]
Denote the vector of generalized velocities of the body by
\[
    v = (\dot{x}, \dot{\theta}, \dot{y})^T.
\]
The power dissipated by friction force~\eqref{eq8} subject to $y = 0$ equals to
\[
    (F, v) = - k \nu \bigl(
        (m_1 + m_2 \cos^2{\theta}) \dot{x}^2 +
        m_2 (\dot{x} \sin{\theta} - l \dot{\theta})^2
    \bigr) \le 0.
\]
On the other hand, the power of the same force dissipated at the point of contact equals to
\[
    F_x \dot{x} = - k \nu \bigl((m_1 + m_2) \dot{x}^2 - l m_2 \dot{x} \dot{\theta} \sin{\theta}\bigr).
\]
Difference between these expressions indicates that the work of friction force~\eqref{eq8}
is not concentrated in the point of contact.
Besides, the expression for the power dissipated at the point of contact is not sign-definite.
Thus, conditions~\ref{cond1}-\ref{cond2} alone turn out to be insufficient for a correct definition
of the dry friction in case of a motion of an absolutely rigid body on an undeformable surface.

\paragraph{Condition of contact interaction}

Denote by $P$ a linear operator that maps an $n$-dimensional vector $v$ of generalized velocities
of the system to an $r$-dimensional vector $v_c$ ($r \le n$) of the velocity of the body at the point of contact
with the supporting surface:
\begin{equation} \label{eq9}
    P v = v_c.
\end{equation}
By definition, operator $P$ is surjective. In particular, matrix $P P^T$ is invertible.

\begin{lemma} \label{lm1}
    For an absolutely rigid body interacting with support at a single point of contact,
    \begin{equation} \label{eq10}
        F = P^T F_c,
    \end{equation}
    where $F_c$ is a Coulomb-Amontons' friction force at the point of contact,
    and $F$ is the corresponding generalized friction force.
\end{lemma}
\begin{proof}
    Under the lemma conditions, all power of the friction force must be dissipated at the point of contact,
    in other words $(F, v) = (F_c, v_c)$.
    From here, taking into account~\eqref{eq9} follows formula~\eqref{eq10}.
\qed\end{proof}

As a consequence from the proved lemma, we obtain a supplementary to conditions~\ref{cond1}--\ref{cond2}
constraint on matrix $\Phi$ from~\eqref{eq1}:
\begin{enumerate}[i.]
    \setcounter{enumi}{2}
    \item \label{cond3}
    \[
        \Phi = P^T \Psi.
    \]
\end{enumerate}

\begin{lemma} \label{lm2}
    Conditions~\ref{cond2} and~\ref{cond3} are equivalent to equality
    \begin{equation} \label{eq11}
        \Phi = P^T \Phi_c P,
    \end{equation}
    where $\Phi_c$ is a non-negative-definite $r \times r$ matrix.
\end{lemma}
\begin{proof}
    Let conditions~\ref{cond2}--\ref{cond3} be met.
    Let us show that $\ker P \subset \ker \Phi \subset \mathbb{R}^n$.
    By contradiction, assume that there exists $\xi$ such that $P \xi = 0$, but $\Phi \xi \ne 0$.
    Then for arbitrary $u \in \mathbb{R}^n$ and $\alpha \in \mathbb{R}$
    \[
        0 \le \bigl(\Phi (u + \alpha \xi), u + \alpha \xi\bigr) =
        u^T \Phi (u + \alpha \xi) =
        u^T \Phi u + \alpha u^T \Phi \xi.
    \]
    But for $u = \Phi \xi$ and sufficiently large in absolute value negative $\alpha$
    the last expression is negative, which leads to a contradiction.
    Hence, $\Phi = \Xi P$, which together with condition~\ref{cond3} implies $\Phi = P^T \Phi_c P$.
    Non-negative-definiteness of $\Phi_c$ follows from~\eqref{eq9} and condition~\ref{cond2}:
    \[
        (\Phi_c v_c, v_c) = (\Phi_c P v, P v) = (P^T \Phi_c P v, v) = (\Phi v, v) \ge 0.
    \]

    Reverse implication is obvious.
\qed\end{proof}

Vector of generalized support reaction $R$ is an image of $P^T$ of normal support reaction $N$
at the point of contact, i.e. $R = P^T N$.
Consequently, the norm of $R$ in formula~\eqref{eq1} can be replaced with the Euclidean length of $N$,
and multiplier $|R| / |N| > 0$ that depends only on generalized coordinates can be included
in not yet defined tensor $\Phi_c$.
At the same time, the norm of projection of vector of generalized velocities $\dot{x}$
onto the point of contact coincides with the Euclidean length of the velocity of this point.
The latter follows from the fact that in the metrics defined by matrix $A$,
the square of length of the specified projection of generalized velocity equals to
the doubled kinetic energy of a material point, proportional to the Euclidean square of velocity.
With that said, holds

\begin{theorem} \label{th1}
    In case of an absolutely rigid body and a single point of contact,
    equality~\eqref{eq1} can be rewritten as
    \begin{equation} \label{eq12}
        F = - |N| \frac{P^T \Phi_c v_c}{|v_c|},
    \end{equation}
    where $|N|$ and $|v_c|$ are the Euclidean lengths of vectors of the force of normal support reaction
    and the velocity of the body at the point of contact, and $\Phi_c$ is a non-negative-definite matrix
    such that
    \begin{equation} \label{eq13}
        \Phi_c^T Q N = \rho N, \qquad Q = P A^{-1} P^T, \quad \rho \in \mathbb{R}.
    \end{equation}
\end{theorem}
\begin{proof}
    Equality~\eqref{eq12} follows from~\eqref{eq1} by substitution~\eqref{eq11} subject to~\eqref{eq9}.
    In order to prove formula~\eqref{eq13}, write down condition~\ref{cond1} in the form
    \[
        (\Phi^T A^{-1}) R = \rho R.
    \]
    By substituting here $R = P^T N$ and~\eqref{eq11}, we get equality
    \[
        P^T \Phi_c^T P A^{-1} P^T N = \rho P^T N, \qquad \rho \in \mathbb{R},
    \]
    equivalent to~\eqref{eq13}.
\qed\end{proof}

\paragraph{Generalized dry friction force for the Painleve problem}

Let us find the general view of friction force $F$ representable in the form~\eqref{eq12} for the Painleve problem.
Herewith,
\[
    v_c = \begin{pmatrix}
        \dot{x} \\
        \dot{y} \\
    \end{pmatrix}, \quad
    N = \begin{pmatrix}
        0       \\
        \lambda \\
    \end{pmatrix}, \quad
    P = \begin{pmatrix}
        1 & 0 & 0 \\
        0 & 0 & 1 \\
    \end{pmatrix}.
\]
From~\eqref{eq6} we get
\[
    Q = \frac{1}{2 m_1 (m_1 + m_2)} \begin{pmatrix}
        2 m_1 + m_2 (1 - \cos{2 \theta}) & - m_2 \sin{2 \theta}             \\
        - m_2 \sin{2 \theta}             & 2 m_1 + m_2 (1 + \cos{2 \theta}) \\
    \end{pmatrix},
\]
and from~\eqref{eq13} it follows that
\begin{equation} \label{eq14}
    \Phi_c = \begin{pmatrix}
        \varkappa_1                             & \varkappa_2 \\
        \dfrac{
            \varkappa_1 m_2 \sin{2 \theta}
        }{
            2 m_1 + m_2 (1 + \cos{2 \theta})
        }                                       & \mu_2       \\
    \end{pmatrix} \qquad (\Phi_c \ge 0).
\end{equation}
The corresponding Lagrange equations taking into account dry friction take the form
\begin{equation}\begin{alignedat}{3} \label{eq15}
    & (m_1 + m_2) \ddot{x} - l m_2 \ddot{\theta} \sin{\theta} && = l m_2 \dot{\theta}^2 \cos{\theta} -
    \smash[b]{ \lambda \frac{
        \varkappa_1 \dot{x} + \varkappa_2 \dot{y}
    }{
        \sqrt{\dot{x}^2 + \dot{y}^2}
    } }, \\
    & l m_2 (\ddot{x} \sin{\theta} - \ddot{y} \cos{\theta} - l \ddot{\theta}) && = l m_2 g \cos{\theta}, \\
    & (m_1 + m_2) \ddot{y} + l m_2 \ddot{\theta} \cos{\theta} && =
    - (m_1 + m_2) g + l m_2 \dot{\theta}^2 \sin{\theta} + \lambda \\
    &&& - \frac{\lambda}{\sqrt{\dot{x}^2 + \dot{y}^2}} \biggl(
        \mu_2 \dot{y} + \frac{
            \varkappa_1 m_2 \dot{x} \sin{2 \theta}
        }{
            2 m_1 + m_2 (1 + \cos{2 \theta})
        }
    \biggr),
\end{alignedat}\end{equation}
System~\eqref{eq15} for different $\lambda$ describes all possible movements of the rod in the Painleve problem.
By setting in this system $\lambda = 0$, we obtain the equations of motion of the rod in the absence of contact.
In turn, letting $y$ with derivatives equal to zero, we get $\lambda$ same as in~\eqref{eq3},
and the equations of motion under continuous contact:
\[\begin{alignedat}{3}
    &&& (m_1 + m_2) \ddot{x} - l m_2 \ddot{\theta} \sin{\theta} = l m_2 \dot{\theta}^2 \cos{\theta} -
    \lambda \varkappa_1 \sigma(\dot{x}), \\
    &&& l m_2 (\ddot{x} \sin{\theta} - l \ddot{\theta}) = l m_2 g \cos{\theta},
\end{alignedat}\]
where the function $\sigma$ is defined in~\eqref{eq4}.
The latter system covers two cases.
The first case takes place when $\dot{x} \ne 0$ and corresponds to a motion with sliding friction.
The second one, where $\dot{x} = 0$, corresponds to rest of the contact point with force of static friction
\[
    F = \lambda \sigma_0 \varkappa_1 = m_2 \cos{\theta} (l \dot{\theta}^2 - g \sin{\theta}),
    \quad
    \sigma_0 = \sigma(0) \in [-1, 1].
\]
Hence the rest condition of the contact point is given by inequality
\[
    |\cos{\theta} (l \dot{\theta}^2 - g \sin{\theta})| \le \frac{\lambda \varkappa_1}{m_2}.
\]
When equality is reached in this formula, rest turns into a slide.

Instant contact with the support (impact) corresponds to $\lambda = + \infty$.
Formally, the equations describing velocity discontinuities can be obtained
by integrating~\eqref{eq15} over an infinitesimal time interval containing the moment of impact,
\[\begin{alignedat}{3}
    & (m_1 + m_2) \Delta{\dot{x}} - l m_2 \Delta{\dot{\theta}} \sin{\theta} && =
    - \varkappa_1 S_x - \varkappa_2 S_y, \\
    & l m_2 (\Delta{\dot{x}} \sin{\theta} - \Delta{\dot{y}} \cos{\theta} - l \Delta{\dot{\theta}}) && = 0, \\
    & (m_1 + m_2) \Delta{\dot{y}} + l m_2 \Delta{\dot{\theta}} \cos{\theta} && =
    I - \mu_2 S_y - \smash[t]{ \frac{
        \varkappa_1 m_2 S_x \sin{2 \theta}
    }{
        2 m_1 + m_2 (1 + \cos{2 \theta})
    } },
\end{alignedat}\]
where $I$ is the impulse of the support reaction force, and $S_x$ and $S_y$ are tangent and normal components
of the impulses of the impact force of friction.
Note that the coefficients $\varkappa_2$ and $\mu_2$ only participate in the equations describing the friction
for impact.

\section{Rigid ball on a plane} \label{sec3}

Consider the problem of the motion of a rigid homogeneous ball of mass $m$ and radius $R$
on a fixed horizontal plane $z = 0$, and compare results with~\cite[\S~6]{Kozlov2010855}.
As generalized coordinates, we select coordinates of the center of the ball $x$, $y$, $z$
and Euler angles relative to the center of the ball $\alpha$, $\beta$, $\gamma$.
The Lagrangian of this problem has the form
\begin{equation} \label{eq16}
    L = m \frac{\dot{x}^2 + \dot{y}^2 + \dot{z}^2}{2}
      + \frac{m R^2}{5} (\dot{\alpha}^2 + \dot{\beta}^2 + \dot{\gamma}^2 + 2 \dot{\alpha} \dot{\gamma} \cos{\beta})
      - m g z
      + \lambda z,
\end{equation}
where $\lambda$ is the Lagrange multiplier, and $f = z - R \ge 0$ is a unilateral constraint.
Thus
\begin{equation} \label{eq17}
    A = \begin{pmatrix}
        & m E_3         & \rvline & 0                        \\
        \hline
        & 0    & \rvline & \frac{2 m R^2}{5} \begin{pmatrix}
                            1           & 0 & \cos{\beta} \\
                            0           & 1 & 0           \\
                            \cos{\beta} & 0 & 1           \\
                          \end{pmatrix}                      \\
        \end{pmatrix}.
\end{equation}
Velocity of the point of the ball touching the surface and the normal support reaction are
\begin{equation} \label{eq18}
    v_c = \begin{pmatrix}
        \dot{x} - \dot{\beta} R \sin{\alpha} + \dot{\gamma} R \cos{\alpha} \sin{\beta} \\
        \dot{y} - \dot{\beta} R \cos{\alpha} - \dot{\gamma} R \sin{\alpha} \sin{\beta} \\
        \dot{z}                                                                        \\
    \end{pmatrix}, \quad
    N = \begin{pmatrix}
        0       \\
        0       \\
        \lambda \\
    \end{pmatrix},
\end{equation}
from where and from~\eqref{eq13}
\begin{equation}\begin{split} \label{eq19}
    P & = \begin{pmatrix}
        1 & 0 & 0 & 0 & - R \sin{\alpha}    & R \cos{\alpha} \sin{\beta}    \\
        0 & 1 & 0 & 0 & - R \cos{\alpha}    & - R \sin{\alpha} \sin{\beta}  \\
        0 & 0 & 1 & 0 & 0                   & 0                             \\
    \end{pmatrix}, \\
    Q & = \frac{1}{m} \begin{pmatrix}
        7 / 2 & 0     & 0 \\
        0     & 7 / 2 & 0 \\
        0     & 0     & 1 \\
    \end{pmatrix}, \quad
    \Phi_c = \begin{pmatrix}
        \varkappa_1 & \varkappa_2 & \varkappa_3 \\
        \mu_1       & \mu_2       & \mu_3       \\
        0           & 0           & \nu         \\
    \end{pmatrix}.
\end{split}\end{equation}
When written in terms of angular velocity vector $\omega = (\omega_x, \omega_y, \omega_z)^T$ defined by
\begin{equation}\begin{split} \label{eq20}
    \dot{\alpha} & = - \omega_z + \cot{\beta} (\omega_x \sin{\alpha} + \omega_y \cos{\alpha}), \\
    \dot{\beta}  & = - \omega_x \cos{\alpha} + \omega_y \sin{\alpha},                          \\
    \dot{\gamma} & = - \csc{\beta} (\omega_x \sin{\alpha} + \omega_y \cos{\alpha}),
\end{split}\end{equation}
instead of the Euler angles, the equations of motion look especially simple:
\begin{equation}\label{eq21}
    \begin{aligned}
        m \ddot{x} & = \frac{
            - \lambda \bigl(\varkappa_1 (\dot{x} - R \, \omega_y) + \varkappa_2 (\dot{y} + R \, \omega_x) + \varkappa_3 \dot{z}\bigr)
        }{
            \sqrt{(\dot{x} - R \, \omega_y)^2 + (\dot{y} + R \, \omega_x)^2 + \dot{z}^2}
        }, \\
        m \ddot{y} & = \frac{
            - \lambda \bigl(\mu_1 (\dot{x} - R \, \omega_y) + \mu_2 (\dot{y} + R \, \omega_x) + \mu_3 \dot{z}\bigr)
        }{
            \sqrt{(\dot{x} - R \, \omega_y)^2 + (\dot{y} + R \, \omega_x)^2 + \dot{z}^2}
        }, \\
        m \ddot{z} & = \lambda - m g - \frac{
            \lambda \nu \dot{z}
        }{
            \sqrt{(\dot{x} - R \, \omega_y)^2 + (\dot{y} + R \, \omega_x)^2 + \dot{z}^2}
        }, \\
        \frac{2}{5} m R \, \dot{\omega}_x & = \frac{
            - \lambda \bigl(\mu_1 (\dot{x} - R \, \omega_y) + \mu_2 (\dot{y} + R \, \omega_x) + \mu_3 \dot{z}\bigr)
        }{
            \sqrt{(\dot{x} - R \, \omega_y)^2 + (\dot{y} + R \, \omega_x)^2 + \dot{z}^2}
        }, \\
        \frac{2}{5} m R \, \dot{\omega}_y & = \frac{
            \lambda \bigl(\varkappa_1 (\dot{x} - R \, \omega_y) + \varkappa_2 (\dot{y} + R \, \omega_x) + \varkappa_3 \dot{z}\bigr)
        }{
            \sqrt{(\dot{x} - R \, \omega_y)^2 + (\dot{y} + R \, \omega_x)^2 + \dot{z}^2}
        }, \\
        \frac{2}{5} m R \, \dot{\omega}_z & = 0.
    \end{aligned}
\end{equation}
Here depending on the value of $\lambda$ the motion of the system can be attributed to one of the following modes:
\begin{itemize}
    \item $\lambda = 0$, there is no contact between the ball and the surface, i.e. $z > R$.
    \item $\lambda \in (0, +\infty)$ corresponds to a motion of the ball on the surface.
    By setting $z = R$, and its derivatives equal to zero, from the third equation we get
    \[
        \lambda = m g.
    \]
    The obtained system is well-known and was studied among others by G.-G.~Coriolis in his famous
    book~\cite{Coriolis1835147}.
    Particularly, the quantities
    \[
        \dot{x} + \frac{2 R}{5} \omega_y, \quad
        \dot{y} - \frac{2 R}{5} \omega_x
    \]
    are first integrals of this system and represent velocity components
    of the so-called \emph{upper center of percussion of the ball},
    i.e. the point of the ball, located at distance $2 / 5$ of the radius vertically above
    the center of the ball (see Fig.~\ref{fig2}).
    \item $\lambda = +\infty$ corresponds to an impact of the ball on the plane,
    when velocity components of the center of the ball and components of the angular velocity suffer a discontinuity.
    At the same time, the horizontal velocity components of the upper center of percussion are preserved.
\end{itemize}

\begin{figure}[!ht]
    \centering\includegraphics{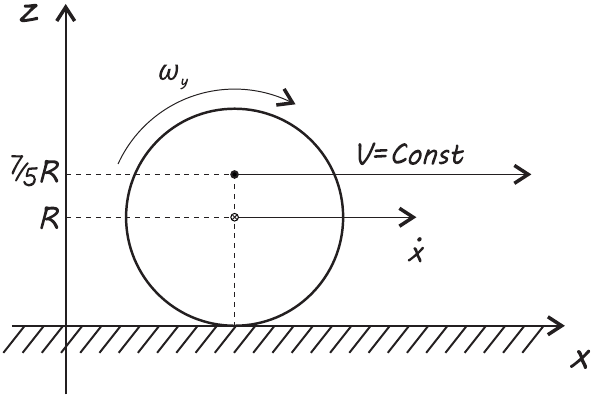}
    \caption{The upper center of percussion}
    \label{fig2}
\end{figure}

The problem of rolling of a heavy homogeneous ball on a horizontal plane was also considered by V.~V.~Kozlov
in~\cite[\S~6]{Kozlov2010855} subject to conditions~\ref{cond1} and~\ref{cond2} only.
The sliding friction force obtained in~\cite{Kozlov2010855}, generally speaking, depends on components
of the angular velocity, that correspond to pivoting and rolling.
Besides, in~\cite{Kozlov2010855} there arise additional moments of forces of rolling and pivoting friction,
which indicates the non-contact character of such a generalized force of dry friction.
In equations~\eqref{eq21} such terms are, obviously, absent.

\section{Motion of a rigid ball in contact with two perpendicular planes} \label{sec4}

In~\cite[\S~7]{Kozlov2010855} conditions~\ref{cond1}--\ref{cond2} are generalized to the case of an arbitrary finite
number of points of contact, which is equivalent to an arbitrary number of constraints (see~\cite[formula~(7.1)]{Kozlov2010855}):
\begin{equation} \label{eq22}
    f_1(x) = 0, \quad \ldots, \quad f_p(x) = 0, \qquad p < n.
\end{equation}
Covectors $\partial f_i/ \partial x$ are assumed linearly independent.
Condition~\ref{cond2} is transferred to this case without changes,
and instead of condition~\ref{cond1}, it is suggested (see~\cite[formula~(7.4)]{Kozlov2010855})
\begin{equation} \label{eq23}
    (\Phi^T A^{-1}) \frac{\partial f_i}{\partial x} = \sum_{ j = 1 }^p c_{i j} \frac{\partial f_j}{\partial x},
    \qquad c_{i j} \in \mathbb R.
\end{equation}
If this condition is met the Lagrange multipliers $\lambda_1, \ldots, \lambda_p$ are uniquely defined,
and as in the case of a single point of contact, they do not depend on $\Phi$.
The following example demonstrates that this requirement is too restrictive.
In fact, in problems with several points of contacts,
we should consider multiple tensors $\Phi_1, \ldots, \Phi_p$, one for each point of contact,
and the Lagrange multipliers can depend on coefficients of friction.

\paragraph{Equations of motion of the ball along the cushion}

Consider equations of motion of a ball on a horizontal half-plane $z = 0$, $y \ge 0$,
bounded by a vertical plane $y = 0$.
Further, these planes will be referred to as \emph{the table} and \emph{the cushion}.
In the absence of contact between the ball and the cushion, we are in the setting of the previously
discussed case.
The impact of the ball on the cushion has also been repeatedly considered in the literature
(see, eg. \cite{Coriolis1835147}).
Below we consider the case of continuous contact of the ball with the cushion and the table,
which is expressed by constraint equations
\begin{equation} \label{eq24}
    y = R, \quad z = R.
\end{equation}

Denote the Lagrange multipliers corresponding to the contacts with the cushion and the table by
$\lambda_\ccushion$ and $\lambda_\ctable$, and the coefficients of sliding friction of the ball
and the surfaces of the cushion and the table by $\mu_\ccushion$ and $\mu_\ctable$.
For the contact with the table, formulas~\eqref{eq17}--\eqref{eq19} obtained in the previous
paragraph are transferred without changes.
For the contact with the cushion,
\[
    v_\ccushion = \begin{pmatrix}
        \dot{x} - \dot{\alpha} R - \dot{\gamma} R \cos{\beta}                          \\
        \dot{y}                                                                        \\
        \dot{z} + \dot{\beta} R \cos{\alpha} + \dot{\gamma} R \sin{\alpha} \sin{\beta}                                                                       \\
    \end{pmatrix}, \quad
    N_\ccushion = \begin{pmatrix}
        0                 \\
        \lambda_\ccushion \\
        0                 \\
    \end{pmatrix},
\]
from where and from~\eqref{eq13}
\[
    P_\ccushion = \begin{pmatrix}
        1 & 0 & 0 & - R & 0              & - R \cos{\beta}            \\
        0 & 1 & 0 & 0   & 0              & 0                          \\
        0 & 0 & 1 & 0   & R \cos{\alpha} & R \sin{\alpha} \sin{\beta} \\
    \end{pmatrix}, \quad
    Q_\ccushion = \frac{1}{m} \begin{pmatrix}
        7 / 2 & 0 & 0     \\
        0     & 1 & 0     \\
        0     & 0 & 7 / 2 \\
    \end{pmatrix}.
\]
For the sake of simplicity, we will assume the frictions of the cushion and the table to be isotropic, i.e.
\begin{equation} \label{eq25}
    \Phi_\ctable = \mu_\ctable E_3, \quad
    \Phi_\ccushion = \mu_\ccushion E_3, \qquad
    (\mu_\ctable, \mu_\ccushion > 0).
\end{equation}
Now using~\eqref{eq12}, we can evaluate generalized friction force $F$.
The Lagrangian for the problem considered differs from~\eqref{eq16} by substitution of $\lambda$ by $\lambda_\ctable$
and addition of term $\lambda_\ccushion y$.
Moving to the angular velocity vector in Lagrange equations using formulas~\eqref{eq20}
and substituting constraint equations~\eqref{eq24}, we obtain the system
\begin{equation}\begin{aligned} \label{eq26}
    m \ddot{x} & = - \frac{
        \lambda_\ctable \mu_\ctable (\dot{x} - R \, \omega_y)
    }{ \sqrt{
        R^2 \, \omega_x^2 + (\dot{x} - R \, \omega_y)^2
    } } - \frac{
        \lambda_\ccushion \mu_\ccushion (\dot{x} + R \, \omega_z)
    }{ \sqrt{
        R^2 \, \omega_x^2 + (\dot{x} + R \, \omega_z)^2
    } }, \\
    0 & = \lambda_\ccushion - \frac{
        \lambda_\ctable \mu_\ctable R \, \omega_x
    }{ \sqrt{
        R^2 \, \omega_x^2 + (\dot{x} - R \, \omega_y)^2
    } }, \\
    0 & = - m g + \lambda_\ctable + \frac{
        \lambda_\ccushion \mu_\ccushion R \, \omega_x
    }{ \sqrt{
        R^2 \, \omega_x^2 + (\dot{x} + R \, \omega_z)^2
    } }, \\
    \frac{2}{5} m R \, \dot{\omega}_x & = - R \, \omega_x \Biggl(
        \frac{
            \lambda_\ctable \mu_\ctable
        }{ \sqrt{
            R^2 \, \omega_x^2 + (\dot{x} - R \, \omega_y)^2
        } } + \frac{
            \lambda_\ccushion \mu_\ccushion
        }{ \sqrt{
            R^2 \, \omega_x^2 + (\dot{x} + R \, \omega_z)^2
        } }
    \Biggr), \\
    \frac{2}{5} m R \, \dot{\omega}_y & = \frac{
        \lambda_\ctable \mu_\ctable (\dot{x} - R \, \omega_y)
    }{ \sqrt{
        R^2 \, \omega_x^2 + (\dot{x} - R \, \omega_y)^2
    } }, \\
    \frac{2}{5} m R \, \dot{\omega}_z & = - \frac{
        \lambda_\ccushion \mu_\ccushion (\dot{x} + R \, \omega_z)
    }{ \sqrt{
        R^2 \, \omega_x^2 + (\dot{x} + R \, \omega_z)^2
    } }. \\
\end{aligned}\end{equation}
From the second and the third equations we find the Lagrange multipliers,
and hence the normal reactions of the cushion and the table:
\begin{equation}\begin{aligned} \label{eq27}
    \lambda_\ccushion & = \frac{
        m g \mu_\ctable R \, \omega_x \sqrt{ R^2 \, \omega_x^2 + (\dot{x} + R \, \omega_z)^2 }
    }{
        \mu_\ccushion \mu_\ctable R^2 \, \omega_x^2 + \sqrt{
            R^2 \, \omega_x^2 + (\dot{x} - R \, \omega_y)^2
        }\, \sqrt{
            R^2 \, \omega_x^2 + (\dot{x} + R \, \omega_z)^2
        }
    }, \\
    \lambda_\ctable & = \frac{
        m g \sqrt{
            R^2 \, \omega_x^2 + (\dot{x} - R \, \omega_y)^2
        }\, \sqrt{
            R^2 \, \omega_x^2 + (\dot{x} + R \, \omega_z)^2
        }
    }{
        \mu_\ccushion \mu_\ctable R^2 \, \omega_x^2 + \sqrt{
            R^2 \, \omega_x^2 + (\dot{x} - R \, \omega_y)^2
        }\, \sqrt{
            R^2 \, \omega_x^2 + (\dot{x} + R \, \omega_z)^2
        }
    }.
\end{aligned}\end{equation}
Note that without account for friction the expression for $\lambda_\ccushion$ would turn to $0$,
and the expression for $\lambda_\ctable$ would come down to $m g$,
which is not sufficient for the description of the motions of the ball,
observed in practice in billiards (see the next paragraph).

Stability of the contacts is determined by $\lambda_\ccushion$ and $\lambda_\ctable$:
the contact is stable when the corresponding Lagrange multiplier is positive.
For $\lambda_\ctable$, this is always true, i.e. continuous contact with the table while moving along the cushion
is never violated.
For $\lambda_\ccushion$, the requirement of positiveness is equivalent to inequality $\omega_x > 0$,
which corresponds to a rotation pressing the ball to the cushion.

Substituting \eqref{eq27} into \eqref{eq26} and performing a change of variables
\[
    \omega_x = \Omega_x, \quad
    \omega_y = \Omega_y + \frac{\dot{x}}{R}, \quad
    \omega_z = \Omega_z - \frac{\dot{x}}{R},
\]
we find the final form of the equations of motion of the ball along the cushion:
\begin{equation}\begin{aligned} \label{eq28}
    \ddot{x} & = - \frac{
        g \mu_\ctable \bigl( \mu_\ccushion \Omega_x \Omega_z - \Omega_y \sqrt{ \Omega_x^2 + \Omega_z^2 } \bigr)
    }{
        \mu_\ccushion \mu_\ctable \Omega_x^2 +
        \sqrt{\Omega_x^2 + \smash[b]{ \Omega_y^2 }}\, \sqrt{\Omega_x^2 + \Omega_z^2}
    }, \\
    \dot{\Omega}_x & = - \frac{
        5 g \mu_\ctable \Omega_x \bigl(\mu_\ccushion \Omega_x + \sqrt{\Omega_x^2 + \Omega_z^2}\bigr)
    }{  2 R \bigl(
        \mu_\ccushion \mu_\ctable \Omega_x^2 +
        \sqrt{\Omega_x^2 + \smash[b]{ \Omega_y^2 }}\, \sqrt{\Omega_x^2 + \Omega_z^2}
    \bigr) }, \\
    \dot{\Omega}_y & = \frac{
        g \mu_\ctable \bigl( 2 \mu_\ccushion \Omega_x \Omega_z - 7 \Omega_y \sqrt{\Omega_x^2 + \Omega_z^2} \bigr)
    }{ 2 R \bigl(
        \mu_\ccushion \mu_\ctable \Omega_x^2 +
        \sqrt{\Omega_x^2 + \smash[b]{ \Omega_y^2 }}\, \sqrt{\Omega_x^2 + \Omega_z^2}
    \bigr) }, \\
    \dot{\Omega}_z & = - \frac{
        g \mu_\ctable \bigl(7 \mu_\ccushion \Omega_x \Omega_z - 2 \Omega_y \sqrt{\Omega_x^2 + \Omega_z^2}\bigr)
    }{ 2 R \bigl(
        \mu_\ccushion \mu_\ctable \Omega_x^2 +
        \sqrt{\Omega_x^2 + \smash[b]{ \Omega_y^2 }}\, \sqrt{\Omega_x^2 + \Omega_z^2}
    \bigr) }, \\
\end{aligned}\end{equation}
It can be easily seen that function
\begin{equation} \label{eq29}
    J = 9 \dot{x} + 2 R (\Omega_y - \Omega_z) = 5 \dot{x} + 2 \omega_y - 2 \omega_z
\end{equation}
is an integral of this system.
Thus, the problem comes down to solving the system of the last three equations~\eqref{eq28}.
\begin{figure}[!ht]
    \centering\includegraphics{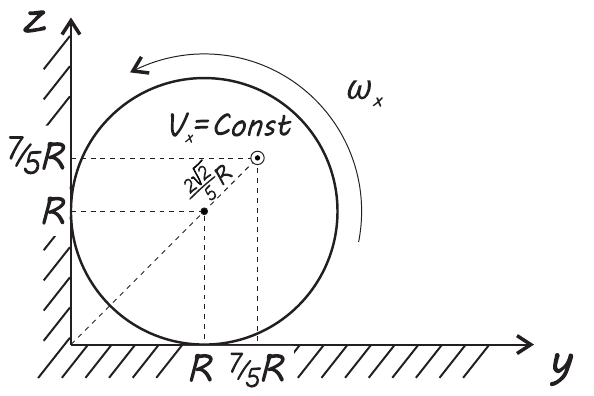}
    \caption{The upper side center of percussion}
    \label{fig3}
\end{figure}
Integral~\eqref{eq29} has the mechanical sense of the $x$-component of velocity of the upper side point
of the ball, located at distance $2 \sqrt{2} / 5$ of the radius from its center along the line,
perpendicular to the intersection of the supporting planes and passing through the center of the ball.
In this regard, this point can be referred to as
\emph{the upper side center of percussion of the ball} (see Fig.~\ref{fig3}).

\paragraph{The "Frenchman" stroke in billiards}

We will look for a solution of system~\eqref{eq28} such that
\begin{equation} \label{eq30}
    \Omega_y = \eta\, \Omega_x, \quad \Omega_z = \zeta\, \Omega_x \qquad
    (\eta, \zeta \in \mathbb{R}).
\end{equation}
Finding the time derivative of~\eqref{eq30} using~\eqref{eq28}, we arrive at the equations for $\eta$ and $\zeta$:
\[\begin{aligned}
    & \bigl(2 \sqrt{1 + \zeta^2} - 5 \mu_\ccushion\bigr) \eta - 2 \mu_\ccushion \zeta = 0, \\
    & \bigl(5 \sqrt{1 + \zeta^2} - 2 \mu_\ccushion\bigr) \zeta + 2 \eta \sqrt{1 + \zeta^2} = 0.
\end{aligned}\]
These equations have the following solutions:
\begin{enumerate}[1.]
    \item \label{sol1}
    $\eta = 0$, $\zeta = 0$.
    This case corresponds to rolling of the ball along the cushion without sliding
    at both points of contact.
    \item \label{sol2}
    $\eta = \pm 2 \sqrt{4 \mu_\ccushion^2 - 1}$, $\zeta = \mp \sqrt{4 \mu_\ccushion^2 - 1}$.
    In order for these solutions to correspond to a real motion,
    the friction coefficient of the ball and the cushion must satisfy inequality $\mu_\ccushion \ge 0.5$.
    Herewith, the direction of vector $\Omega = (\Omega_x, \Omega_y, \Omega_z)^T$ does not change until
    the moment $T > 0$, when the spin (called \emph{english} in billiards) pressing the ball
    to the cushion disappears due to friction and the ball goes over to the rolling regime
    described previously.
    \item \label{sol3}
    $\eta = \pm \frac{1}{4} \sqrt{\mu_\ccushion^2 - 4}$, $\zeta = \mp \frac{1}{2} \sqrt{\mu_\ccushion^2 - 4}$.
    In order for these solutions to correspond to a real motion,
    the friction coefficient of the ball and the cushion must satisfy inequality $\mu_\ccushion \ge 2$.
    In practice, such large friction coefficients do not arise, at least in billiards,
    thence we will not stop on consideration of this case.
\end{enumerate}
Write down explicit dependencies of the velocities on time for solution~\ref{sol2}:
\begin{equation}\label{eq31}
    \begin{aligned}
        \dot{x}(t) & = v_0 \pm \frac{
            5 \mu_\ctable g t \sqrt{4 \mu_\ccushion^2 - 1}
        }{
            2 \sqrt{16 \mu_\ccushion^2 - 3} + \mu_\ctable
        }, \\
        \omega_x(t) & = \omega_{x, 0} - \frac{
            15 \mu_\ctable g t
        }{
            2 R \Bigl( 2 \sqrt{16 \mu_\ccushion^2 - 3} + \mu_\ctable \Bigr)
        }, \\
        \omega_y(t) & = \frac{v_0}{R} \pm 2 \Biggl(
            \omega_{x, 0} - \frac{
                5 \mu_\ctable g t
            }{
                R \Bigl( 2 \sqrt{16 \mu_\ccushion^2 - 3} + \mu_\ctable \Bigr)
            }
        \Biggr) \sqrt{4 \mu_\ccushion^2 - 1}, \\
        \omega_z(t) & = - \frac{v_0}{R} \mp \Biggl(
            \omega_{x, 0} - \frac{
                5 \mu_\ctable g t
            }{
                2 R \Bigl( 2 \sqrt{16 \mu_\ccushion^2 - 3} + \mu_\ctable \Bigr)
            }
        \Biggr) \sqrt{4 \mu_\ccushion^2 - 1}, \\
    \end{aligned}
\end{equation}
for $t \in [0, T)$, and
\[\begin{aligned}
    \dot{x}(t) & = v_0 \pm \frac{2}{3}\, R\, \omega_{x, 0} \sqrt{4 \mu_\ccushion^2 - 1}, \\
    \omega_x(t) & = 0, \\
    \omega_y(t) & = \frac{v_0}{R} \pm \frac{2}{3}\, \omega_{x, 0} \sqrt{4 \mu_\ccushion^2 - 1}, \\
    \omega_z(t) & = - \frac{v_0}{R} \mp \frac{2}{3}\, \omega_{x, 0} \sqrt{4 \mu_\ccushion^2 - 1}, \\
\end{aligned}\]
for $t \in [T, +\infty)$, where $v_0$ and $\omega_{x, 0}$ are the initial values of the longitudinal velocity
of the ball center and side spin, and
\[
    T = \frac{
        2 \omega_{x, 0} R \Bigl( \mu_\ctable + 2 \sqrt{16 \mu_\ccushion^2 - 3} \Bigr)
    }{
        15 \mu_\ctable g
    }.
\]
Initial values $\omega_{y, 0}$ and $\omega_{z, 0}$ are obtained by substituting $t = 0$ into~\eqref{eq31}:
\[
    \omega_{y, 0} = \frac{v_0}{R} \pm 2 \omega_{x, 0} \sqrt{4 \mu_\ccushion^2 - 1}, \quad
    \omega_{z, 0} = - \frac{v_0}{R} \mp \omega_{x, 0} \sqrt{4 \mu_\ccushion^2 - 1}.
\]
The equations obtained describe one of the most recognizable techniques in Russian pyramid,
called \emph{the "Frenchman"}.
A similar, but slightly different in execution, stroke occurs in artistic pool and is
sometimes called \emph{"Rocket masse"}.

\section{Tensors of dry friction in multi-contact case} \label{sec5}

In this section we suggest a multi-contact generalization of condition~\ref{cond1} alternative to~\eqref{eq23},
that is appropriate for the description of motions like in the previous example.

Consider a rigid body touching the support at $p$ points of contact according to constraint equations~\eqref{eq22}.
Denote by $v_i$, $i = 1, \ldots, p$,
the velocity of the body at the point of contact with the $i$-th supporting surface $f_i(x) = 0$,
by $P_i: \mathbb{R}^n \to \mathbb{R}^r$, $r \le n$,
the operator that maps generalized velocity $v$ to $v_i$,
by $N_i$ the normal support reaction of the $i$-th surface,
and by $\Phi_i$ the tensor of dry friction at the $i$-th contact point.
Also introduce covectors $n_i$ such that $N_i = \lambda_i n_i$, or equivalently
\[
    P_i^T n_i = \frac{\partial f_i}{\partial x}.
\]
Then according to~\eqref{eq12}, the generalized dry friction force takes the form
\[
    F = - \sum_{j = 1}^p \lambda_j |n_j| \frac{P_j^T \Phi_j v_j}{|v_j|}.
\]
Following V.~V.~Kozlov's considerations for the case of a single point of contact
(see~\cite[formulas~(2.10)--(2.13)]{Kozlov2010855}),
rewrite the Lagrange equations as
\[
    \ddot{x} = A^{-1} X + \sum_{j = 1}^p \lambda_j A^{-1} P_j^T \biggl(
        n_j - |n_j| \frac{\Phi_j v_j}{|v_j|}
    \biggr),
\]
where $X$ is a known function of $\dot{x}$ and $x$.
Now, multiplying scalarly this equation by $\partial f_i / \partial x = P_i^T n_i$, we get
a system on $\lambda_i$:
\begin{equation} \label{eq32}
    \sum_{j = 1}^p \lambda_j \biggl( P_i^T n_i, A^{-1} P_j^T \biggl[
        n_j - |n_j| \frac{\Phi_j v_j}{|v_j|}
    \biggr] \biggr)
    = Z_i - \biggl( P_i^T n_i, A^{-1} X \biggr),
\end{equation}
where $Z_i$ are also known functions of $\dot{x}$ and $x$,
\[
    Z_i = \biggl( \frac{\partial f_i}{\partial x}, \ddot{x} \biggr)
    = - \biggl( \frac{d}{dt} \frac{\partial f_i}{\partial x}, \dot{x} \biggr),
\]
and the right-hand side of~\eqref{eq32} does not contain $\lambda_j$.
Thus, the absence of Painleve paradoxes in the Lagrange equations is equivalent to unique solvability
of~\eqref{eq32} with respect to $\lambda_i$.

\begin{definition} \label{df1}
    We call a set of dry friction tensors $\Phi_i$, $i = 1, \ldots, p$ \emph{regular} if
    for any $v$ satisfying constraint equations
    \[
        \biggl( \frac{\partial f_1}{\partial x}, v \biggr) = \cdots =
        \biggl( \frac{\partial f_p}{\partial x}, v \biggr) = 0,
    \]
    and for any $k > 0$
    \begin{equation} \label{eq33}
        \det \Bigl\{ \Bigl( n_i, P_i A^{-1} P_j^T \Bigl[
            n_j - k \Phi_j v_j
        \Bigr] \Bigr) \Bigr\} > 0.
    \end{equation}
\end{definition}
Condition~\eqref{eq33} expresses a natural property of regular dry friction tensors
that scale conversions of the friction coefficients should not lead to Painleve paradoxes.
In particular, assuming parameter $k$ sufficiently large and leaving the highest terms in $k$ only,
we get a simple necessary condition of regularity:
\[
    \det \Bigl\{ \Bigl( n_i, P_i A^{-1} P_j^T \Phi_j v_j \Bigr) \Bigr\} \ge 0,
\]
where in case of odd $p$ the inequality turns to equality.

\begin{example} \label{ex1}
    In case of a single point of contact, \eqref{eq33} takes the form
    \[
        \Bigl( n_c, P A^{-1} P^T \Phi_c v_c \Bigr) = 0,
    \]
    which is equivalent to condition~\ref{cond1}.
    Consequently, tensors $\Phi_c$ from \eqref{eq14} for the Painleve problem
    and \eqref{eq19} for the problem of a ball on a plane
    are regular.
\end{example}

\begin{example} \label{ex2}
    Let for the problem of a ball moving along the cushion
    \[
        \Phi_\ccushion = \begin{pmatrix}
            \varkappa_1 & \varkappa_2 & \varkappa_3 \\
            0           & \mu_2       & 0           \\
            0           & \nu_2       & \nu_3       \\
        \end{pmatrix},
        \quad
        \Phi_\ctable = \begin{pmatrix}
            \varkappa_4 & \varkappa_5 & \varkappa_6 \\
            0           & \mu_5       & \mu_6       \\
            0           & 0           & \nu_6       \\
        \end{pmatrix},
    \]
    where $\varkappa_{1, 4}, \mu_{2, 5}, \nu_{3, 6} \ge 0$.
    Then~\eqref{eq33} turns into
    \[
        \frac{
            1 + k^2 R^2 \mu_5 \nu_3 \omega_x^2
        }{
            m^2
        } > 0,
    \]
    which is always true.
    In particular, isotropic tensors $\Phi_\ctable$ and $\Phi_\ccushion$ from~\eqref{eq25} are regular.
\end{example}

\bibliography{preprint}

\end{document}